\documentclass[11pt]{amsart}

\usepackage{hyperref}
\usepackage{amsfonts}
\usepackage{amssymb}
\usepackage{graphicx}
\usepackage{amsmath,mathtools}
\usepackage{latexsym,enumerate}
\usepackage{amscd}
\usepackage{xypic}
\usepackage{mathrsfs}
\usepackage{enumitem} 
\usepackage{braket}
\usepackage[margin=1in]{geometry}
\usepackage{amsthm}
\usepackage{accents}

\allowdisplaybreaks

\newtheorem{prop}{Proposition}[section]
\newtheorem{theo}[prop]{Theorem}
\newtheorem{lemm}[prop]{Lemma}
\newtheorem{coro}[prop]{Corollary}
\newtheorem{rema}[prop]{Remark}
\newtheorem{conj}[prop]{Conjecture}

\theoremstyle{definition}

\newcommand{\RR}{\mathbb{R}}
\renewcommand{\SS}{\mathbb{S}}

\newcommand{\bangle}[1]{\left\langle #1 \right\rangle}

\newcommand{\abs}[1]{\left|{#1}\right|}

\let\oldmarginpar\marginpar
\renewcommand\marginpar[1]{\-\oldmarginpar[\raggedleft\footnotesize #1]%
{\raggedright\footnotesize #1}}

\usepackage{graphicx}

\allowdisplaybreaks

\title{Uniqueness of the bowl soliton}

\author{Robert Haslhofer}
\date{August 15, 2014}
\thanks{Supported by NSF Grant DMS-1406394.}

\begin{document}

\maketitle

\begin{abstract}
We prove that any translating soliton for the mean curvature flow which is noncollapsed and uniformly 2-convex must be the rotationally symmetric bowl soliton.
In particular, this proves a conjecture of White and Wang, in the 2-convex case in arbitrary dimension.
\end{abstract}

\section{Introduction}

A hypersurface $M^n\subset\mathbb{R}^{n+1}$ (complete, embedded, oriented) is called a \emph{translating soliton} if its mean curvature $H$ and its normal vector $\nu$ are related by the equation
\begin{equation}\label{eq_soliton}
H=\langle V,\nu\rangle,
\end{equation}
for some $0\neq V\in \mathbb{R}^{n+1}$. Solutions of \eqref{eq_soliton} correspond to translating solutions $\{M_t=M+tV\}_{t\in\mathbb{R}}$ of the mean curvature flow,
\begin{equation}
\partial_t x=H\nu.
\end{equation}
Translating solitons play a key role in the study of slowly forming singularities, see e.g. \cite{Ham_harnack,AngVel,HS1,HS2,White_nature}, and have received a lot of attention in the last 20 years.\\

It is not hard to see that there exists a unique solution (up to rigid motion) of \eqref{eq_soliton} which is rotationally symmetric and strictly convex \cite{AltWu}.
For $n=1$, this is the \emph{grim reaper} \cite{Mullins}, given by the explicit formula $y=-\log\cos(x)$, $x\in(-\pi/2,\pi/2)$.
For $n\geq 2$, which we assume from now on, the solution roughly looks like a paraboloid, and is usually called the \emph{bowl soliton}.\\

A well known problem concerns the uniqueness of solutions of \eqref{eq_soliton}, see e.g. White \cite{White_nature}: Conjecture 2 on page 133, and the unnumbered remark at the bottom of the page 133.
A very important contribution was made by X.~Wang \cite{wang_convex}, who proved that for $n=2$ any entire convex solution must be the rotationally symmetric bowl soliton,
but that for $n\geq 3$ there exist entire strictly convex solution that are not rotationally symmetric.
However, the most relevant solutions of \eqref{eq_soliton} are of course the ones that actually arise as singularity models for the mean curvature flow, and it is unknown whether or not Wang's solutions can actually arise as singularity models.

\begin{conj}[{White \cite[p. 133]{White_nature}, Wang \cite[p. 1237]{wang_convex}}]\label{conjecture}
 All translating solitons that arise as blowup limit of a mean convex mean curvature flow must be rotationally symmetric.
\end{conj}

The conjecture is motivated by the deep regularity and structure theory for mean convex mean curvature flow due to White \cite{White_size,White_nature}, see also Haslhofer-Kleiner \cite{HK}.
In particular, it is known that the only \emph{shrinking} solitons that can occur as blowup limits in the mean convex case are the round shrinking cylinders $S^j\times \RR^{n-j}$.
Recently, Colding-Minicozzi \cite{CM_lojas} proved that even the axis of such a cylindrical tangent flow is unique, i.e. independent of the sequence of rescaling factors.\\

In general, a very important feature of blowup limits of a mean convex mean curvature flow is that they are always noncollapsed \cite{White_size,sheng_wang,andrews1,HK}. 
To address Conjecture \ref{conjecture} we can thus focus on solutions of \eqref{eq_soliton} that are \emph{$\alpha$-noncollapsed},
i.e. solutions with positive mean curvature such that at each $p\in M$ the inscribed radius and the outer radius are at least $\tfrac{\alpha}{H(p)}$.
Indeed, by the references quoted above, see e.g. \cite[Thm. 3]{andrews1} and \cite[Thm. 1.14]{HK},
every blowup limit of a mean convex flow is $\alpha$-noncollapsed for some $\alpha>0$.\footnote{In fact, a more quantitative analysis shows that the optimal constant $\alpha$ is at least $1$ \cite[Cor 1.5]{HK_rem}.}\\


While we do not know at the moment how to address Conjecture \ref{conjecture} in full generality (see however Remark \ref{rem_generalcase}),
in the present article we manage to prove it in an important special case, namely the uniformly $2$-convex case in arbitrary dimension, see Corollary \ref{main_cor}.
We recall that an oriented hypersurface is called \emph{uniformly 2-convex}, if it is mean convex and satisfies
\begin{equation}
 \lambda_1+\lambda_2\geq \beta H
\end{equation}
for some $\beta>0$, where $\lambda_1\leq\lambda_2\leq\ldots\leq \lambda_n$ denotes the principal curvatures.
Uniform 2-convexity is preserved under mean curvature flow, and arises e.g. in the construction of mean curvature flow with surgery by Huisken-Sinestrari \cite{HS3},
see also Haslhofer-Kleiner \cite{HK_surgery} and Brendle-Huisken \cite{BrendleHuisken}.
In this setting of flows with surgery, there can be some high curvature regions, e.g. regions like a degenerate neckpinch, that are modelled on translating solitons that must be $\alpha$-noncollapsed and uniformly $2$-convex.
It turns out, that in such a situation one can always find an almost round cylindrical region at controlled distance from the tip, where one can perform the surgery,
and thus (somewhat surprisingly) one can prove the existence of mean curvature flow with surgery without actually knowing whether or not the translating solitons that occur are rotationally symmetric.
However, in addition to existence, one of course wants to know how the flow with surgery looks like.
In particular, one wonders whether all translating solitons that occur as singularity models in this context are actually rotationally symmetric.\\

Our main theorem of the present article is the following.

\begin{theo}\label{main_thm}
Any solution of the translating soliton equation \eqref{eq_soliton} which is $\alpha$-noncollapsed and uniformly 2-convex must be the rotationally symmetric bowl soliton.
\end{theo}

Theorem \ref{main_thm} holds in arbitrary dimension.

\begin{rema} In the special case $n=2$, the uniform 2-convexity assumption is of course automatic for $\beta=1$.
 In particular, this yields a shorter proof of the 2-dimensional uniqueness result of Wang \cite[Thm. 1.1]{wang_convex}, under somewhat different assumptions (Wang assumes that the solution is convex and that it can be written as an entire graph. We assume that the solution $\alpha$-noncollapsed).
\end{rema}

As an immediate consequence of Theorem \ref{main_thm}, we obtain an affirmative answer to the conjecture of White and Wang (Conjecture \ref{conjecture}), in the 2-convex case in arbitrary dimension.

\begin{coro}\label{main_cor}
 The only translating soliton that can arise as blowup limit of a mean curvature flow of closed embedded 2-convex hypersurfaces, is the rotationally symmetric bowl soliton.
\end{coro}

As another consequence of Theorem \ref{main_thm} we obtain a classification of the translating solitons that can arise as models for high curvature regions in the mean curvature flow with surgery. We state this in the language of the canonical neighborhood theorem \cite[Thm. 1.22]{HK_surgery}.

\begin{coro}\label{cor_surgery}
 The only translating soliton that can arise as canonical neighborhood in the mean curvature flow with surgery, is the rotationally symmetric bowl soliton.
\end{coro}

Let us now discuss some related results. In a very important recent paper \cite{Bre1}, Brendle proved uniqueness of translating solitons for the Ricci flow
in dimension three assuming only $\kappa$-noncollapsing, as suggested by Perelman \cite{Perelman}.
In \cite{Bre2}, Brendle extended his result to higher dimensions, assuming that the soliton has positive sectional curvature and is asymptotically cylindrical.
Using similar techniques, Chodosh \cite{Chodosh} and Chodosh-Fong \cite{ChodoshFong}, obtained uniqueness results for asymptotically conical Ricci expanders.
Using ideas centering around the backwards uniqueness of the heat equation,
L.~Wang \cite{LuWang1,LuWang2} and Wang-Kotschwar \cite{WangKotschwar} proved uniqueness results for asymptotically conical / cylindrical solitons for the mean curvature flow and Ricci flow.
Using Alexandrov's reflection principle, Martin, Savas-Halilaj and Smoczyk proved a uniqueness result for translating solitons for the mean curvature flow, imposing strong asymptotic asumptions \cite{MSHS}: see Theorem B and the asymptotic asumptions in equation (3.1) on page 15.
The work \cite{CSS} by Clutterbuck-Schn\"urer-Schulze also implicitly includes uniqueness result under related assumptions.\\ 

Our proof of Theorem \ref{main_thm} follows a scheme inspired by the recent work of Brendle \cite{Bre1},
and uses some estimates for $\alpha$-noncollapsed flows from Haslhofer-Kleiner \cite{HK}.
The present article seems to be the first one, where Brendle's scheme of proof is implemented for a geometric flow other than Ricci flow.
Another feature of our proof is that we incorporate the ambient euclidean space into our set up right from the beginning; this allows us to give a quite short and efficient argument.\\

Let us now outline the main steps of our proof, pretending $n=2$ and $V=\tfrac{\partial}{\partial z}$ for ease of notation. In Section \ref{sec_asympt}, we study the asymptotic geometry of $M$.
Using some estimates from \cite{HK}, we prove that $H\sim z^{-1/2}$ and that suitable rescalings at infinity are modelled on the round shrinking cylinder $S^{1}\times\mathbb{R}$,
respectively $S^{n-1}\times\mathbb{R}$ in arbitrary dimension (Proposition \ref{prop_asympt}).
In Section \ref{sec_ell_est}, we consider the function $f_{R}=\langle R,\nu\rangle$ where $R=x_1\partial_{x_2}-x_2\partial_{x_1}$ is a rotation centered at the origin.
The function $f_R$ satisfies the same linear elliptic equation as the mean curvature $H$, and we prove a weighted estimate for it (Proposition \ref{prop_ell_est}).
In Section \ref{sec_par_est}, we prove a decay estimate for the corresponding linear parabolic equation on the round shrinking cylinder (Proposition \ref{lemma_decay}).
Finally, in Section \ref{sec_blowdown}, we carry out a blowdown and centering argument, which is the key step of our proof.
Namely, to prove rotational symmetry we would like to find a point $\bar{x}\in \mathbb{R}^2$ such that the rotation function $f_{R_{\bar{x}}}=\langle (x_1-\bar{x}_1)\partial_{x_2}-(x_2-\bar{x}_2)\partial_{x_1},\nu\rangle$ centered at $\bar{x}\in\mathbb{R}^2$ vanishes identically.
To this end, we consider the function 
\begin{equation}
 B(h):=\inf_{x\in\mathbb{R}^2}\sup_{\{z=h\}}\abs{f_{R_{x}}}=\inf_{T\in\mathbb{R}^2}\sup_{\{z=h\}}\abs{f_R-\langle T,\nu\rangle},
\end{equation}
where we observe that the infimum over different centers
can be alternatively written as an infimum subtracting off translation functions; this is important to get a better rate of decay in the estimate for the parabolic equation, c.f.
Proposition \ref{lemma_decay}.
We then fix a suitable small constant $\tau>0$, and assume towards a contradiction that $B(h)>0$ for $h$ large.
On the one hand, since $B(h)\leq O(h^{1/2})$ we can find a sequence $h_m\to\infty$ such that
\begin{equation}
 B(h_m)\leq 2\tau^{-1/2}B(\tau h_m).
\end{equation}
On the other hand, using the estimates that we just described (in particular the fact that we can find a suitable blowdown which is a shrinking cylinder, and Proposition \ref{lemma_decay} which gives a decay estimate on the cylinder) we argue that
\begin{equation}
 B(\tau h_m)\leq \tfrac14 \tau^{1/2}B(h_m)
\end{equation}
for $m$ large. This gives the desired contradiction, and concludes our outline of the proof.\\

Finally, here is some partial progress towards the general case of Conjecture \ref{conjecture}.

\begin{rema}[General case]\label{rem_generalcase}
Let $M\subset\RR^{n+1}$ be a translating soliton which is $\alpha$-noncollapsed. By the convexity estimate \cite[Thm. 1.10]{HK} the soliton must be convex. Let $k\in\{2,\ldots,n\}$ be the smallest integer such that
\begin{equation}\label{eqn_defofk}
 \inf_M\frac{\lambda_1+\ldots \lambda_k}{H}>0.
\end{equation}
Let $p_j$ be a sequence of points going to infinity such that $\tfrac{\lambda_1+\ldots+\lambda_{k-1}}{H}(p_j)\to 0$.
By the global convergence theorem \cite[Thm. 1.12]{HK} after rescaling by $H^{-1}(p_j)$ we can pass to a smooth limit, which we call the asymptotic soliton. By the strict maximum principle, this limit must split off $k-1$ lines.
Together with \eqref{eqn_defofk} and \cite[Lem. 3.14]{HK} it follows that the asymptotic soliton is a round cylinder $S^{n+1-k}\times\RR^{k-1}$.
The missing step is to use the splitting of the asymptotic soliton to obtain splitting properties for the original soliton $M$.
If we could show that $M$ splits off $k-2$ lines, then by Theorem \ref{main_thm} we could conclude that $M$ is isomteric to $\RR^{k-2}\times \textrm{Bowl}_{n-k+2}$.
\end{rema}

\section{Asymptotic geometry}\label{sec_asympt}
Throughout this article $M^n\subset \mathbb{R}^{n+1}$ denotes a translating soliton which is $\alpha$-noncollapsed and uniformly $2$-convex.
Without loss of generality, we can assume that $V=\tfrac{\partial}{\partial z}$, where $z$ denotes the last coordinate in $\mathbb{R}^{n+1}$.
We recall the soliton equation,
\begin{equation}
 H=\langle V,\nu\rangle.
\end{equation}
We can decompose $V$ into its normal part $V^\perp=\langle V,\nu\rangle \nu$ and its tangential part $\bar{V}=V-V^\perp$.
Note that
\begin{equation}\label{eq_sumofsquares}
H^{2} + |\bar{V}|^{2} = 1,
\end{equation}
in particular $H\leq 1$.
We recall the evolution equation for the mean curvature,
\begin{equation}\label{eq_evolh}
- \nabla_{\bar{V}} H = \Delta H + |A|^{2}H.
\end{equation}
Moreover, differentiating the soliton equation we obtain the identity
\begin{equation}\label{eq_dh}
\nabla H=-A(\bar{V},\cdot).
\end{equation}
By \cite[Cor. 2.15]{HK} the soliton can be written in the form $M=\partial K$, where $K$ is a convex domain (in particular connected). In fact, $K$ must be strictly convex.
Indeed, if $\lambda_1$ vanished at some point, then by the strict maximum principle $M$ would split off a line. Together with the uniform 2-convexity and \cite[Lem. 3.14]{HK} this would imply that
$M=S^{n-1}\times\mathbb{R}$, which is absurd.\\

Fix a point $p_0\in M$. Note that $M$ must be noncompact, by comparison with round spheres.

\begin{lemm}\label{lemma_htozero}
 We have $H(p)\to 0$ as $\abs{p-p_0}\to\infty$.
\end{lemm}

\begin{proof}
 If not, there is sequence of points $p_j\in\partial K$ going to infinity such that $\lim\inf_{j\to\infty}H(p_j)>0$.
After passing to a subsequence, we can assume that $\frac{p_j-p_0}{\abs{p_j-p_0}}$ converges to some direction $\omega\in S^n$. Recall that $\abs{H}\leq 1$.
Let $\{\hat{M}_t\}_{t\in\mathbb{R}}$ be the sequence of mean curvature flows obtained from $\{M_t=M+tV\}_{t\in\mathbb{R}}$ by shifting $p_j$ to the origin,
and pass to a subsequential limit $\{\hat{M}_t^\infty\}_{t\in\mathbb{R}}$.
Since $K$ is convex and $\frac{p_j-p_0}{\abs{p_j-p_0}}\to \omega$, the limit contains a line.
Together with the uniform 2-convexity and \cite[Lem. 3.14]{HK} this implies that $\{\hat{M}_t^\infty\}_{t\in\mathbb{R}}$ is a family of round shrinking cylinders $S^{n-1}\times\mathbb{R}$.
This contradicts the fact that $\{\hat{M}_t^\infty\}_{t\in\mathbb{R}}$ is an eternal flow with bounded curvature.
\end{proof}

By Lemma \ref{lemma_htozero} we can find a point $o\in M$, where $H$ attains its maximum. After translating coordinates we can assume that $o$ is the origin in $\mathbb{R}^{n+1}$.
By equation \eqref{eq_dh} and strict convexity the vector $\bar{V}$ must vanish at $o$. Thus, $H(o)=1$ and $T_oM=\mathbb{R}^n\subset\mathbb{R}^{n+1}$ is horizontal. In particular, $M$ is contained in the upper half plane.

\begin{lemm}\label{lemma_lowerbound}
 There exists a constant $c>0$ such that $\inf_{\{z=h\}}H\geq ch^{-1/2}$ for $h$ large.
\end{lemm}

\begin{proof}
Choose $c=\alpha^2/(8n)$. Assume that for some large $h$ there is a point $p\in \{z=h\}$ with $H(p)< ch^{-1/2}$.
By the $\alpha$-noncollapsing condition we can find an interior ball of radius at least $\tfrac{\alpha h^{1/2}}{c}$ tangent at $p$.
By Lemma \ref{lemma_htozero} and the soliton equation the vector $\nu(p)$ is almost horizontal. Thus, it takes time at least
\begin{equation}
 T=\frac12\frac{\alpha^2}{2nc}h= 2h
\end{equation}
until the interior ball leaves the halfspace $\{z\leq h\}$.
On the other hand, since $M$ is contained in the upper half plane and $\{M_t=M+tV\}_{t\in\mathbb{R}}$ moves in $z$-direction with unit speed, the ball must leave the halfspace $\{z\leq h\}$ in time at most $h$; this is a contradiction.
\end{proof}

\begin{prop}\label{prop_asympt}
For any sequence $h_m\to\infty$, the sequence of flows $\{\hat{M}^m_t\}_{t\in(-\infty,1)}$ obtained by translating $p_m=(0,h_m)\in\mathbb{R}^{n+1}$ to the origin and parabolically rescaling by $\lambda_m=h_m^{-1/2}$, i.e.
\begin{equation}
\hat{M}^m_t=\lambda_m\cdot(M_{\lambda_m^{-2} t}-p_m),
\end{equation}
converges to a family of round shrinking cylinders
\begin{equation}
C_t=S^{n-1}_{r(t)}\times\mathbb{R},
\end{equation}
 where $r(t)=\sqrt{2(n-1)(1-t)}$.
\end{prop}

\begin{proof}
 Write $M=\partial K$ as before, and let $\hat{K}^m_t=\lambda_m\cdot(K_{\lambda_m^{-2} t}-p_m)$. Note that $\hat{K}^m_t$ contains the origin for all $t<1$. By \cite[Thm. 1.14]{HK} we can thus find a subsequence that converges smoothly to a limit
$\{\hat{K}^\infty_t\}_{t\in(-\infty,1)}$, whose time slices are convex and nonempty. By Lemma \ref{lemma_lowerbound} the limit is nontrivial, i.e. $\hat{K}^\infty_t\neq \RR^{n+1}$ and $\partial \hat{K}^\infty_t$ has strictly positive mean curvature for all $t<1$ .
Note that the vertical line through the origin is contained in $\hat{K}^\infty_t$ for all $t<1$. Thus,
by the uniform 2-convexity and \cite[Lem. 3.14]{HK} the flow $\{\hat{K}^\infty_t\}_{t\in(-\infty,1)}$ must be a family of round shrinking cylinders $\{C_t=S^{n-1}_{r(t)}\times\mathbb{R}\}_{t\in(-\infty,1)}$.
Since $\hat{K}^m_{1+\varepsilon}\subset \{z\geq \varepsilon h_m^{1/2}\}$ for every $\varepsilon>0$, the cylinders become extinct at time $T=1$, and thus their radius is given by the formula $r(t)=\sqrt{2(n-1)(1-t)}$.
Finally, by uniqueness of the limit the subsequential convergence is actually convergence of the full sequence.
\end{proof}

\begin{coro}
For $z\to\infty$, the mean curvature satisfies the estimate
\begin{equation}
H=(\tfrac{n-1}{2})^{1/2}z^{-1/2}+o(z^{-1/2}).
\end{equation}
\end{coro}

\section{Weighted estimate for the rotation functions}\label{sec_ell_est}

For any translation vector field $T$ and any rotation vector field $R$ in $\mathbb{R}^{n+1}$, we consider the functions $f_T=\bangle{T,\nu}$ and $f_R=\bangle{R,\nu}$, respectively.
Since mean curvature flow is invariant under isometries of the ambient space, these functions satisfy the elliptic equation
\begin{equation}\label{eq_elliptic}
 (\Delta_z+\abs{A}^2)f=0,
\end{equation}
where $\Delta_z=\Delta+\nabla_{\bar{V}}$ is the drift Laplacian. For example, in the case $T=V$ we have $f_V=\bangle{V,\nu}=H$,
which shows that \eqref{eq_elliptic} then reduces to \eqref{eq_evolh} and moreover gives the asymptotics $f_V\sim z^{-1/2}$.

\begin{prop}\label{prop_ell_est}
 For all $h>0$, we have the weighted estimate
\begin{equation}
 \sup_{\{z\leq h\}}\abs{\frac{f_R}{H}}\leq \sup_{\{z= h\}}\abs{\frac{f_R}{H}}.
\end{equation}
\end{prop}

\begin{proof}
Pick the smallest $\theta\geq 0$ such that $u:=\theta H-f_R\geq 0$ in $\{z\leq h\}$. If $\theta=0$, then $f_R\leq 0$. Assume now $\theta>0$. Since
\begin{equation}
 (\Delta_z+\abs{A}^2)u=0,
\end{equation}
the minimum of $u$ in $\{z\leq h\}$ is attainded at a point $p\in \{z= h\}$. By minimality of $\theta$, the minimum must be zero.
Thus,
\begin{equation}
 \frac{f_R}{H}\leq \theta=\frac{f_R}{H}|_p\leq \sup_{\{z=h\}}\abs{\frac{f_R}{H}}.
\end{equation}
Repeating the same argument with $f_R$ replaced by $-f_R$, this proves the proposition.
\end{proof}

\section{Decay estimate on the cylinder}\label{sec_par_est}

\begin{prop}\label{lemma_decay}
 Consider the family of shrinking cylinders $\{C_t=S^{n-1}_{r(t)}\times \mathbb{R}\subset \mathbb{R}^{n+1}\}_{t\in(0,1)}$, where $r(t)=\sqrt{2(n-1)(1-t)}$.
Let $f=\{f(t)\}_{t\in(0,1)}$ be a family of functions on $C_t$ satisfying the parabolic equation
\begin{equation}
  \partial_t {f}=(\Delta_{C_t}+\abs{A_{C_t}}^2) f.
\end{equation}
Suppose that $f$ is invariant under translations along the axis of the cylinder, that
\begin{equation}\label{decay_assump_average}
\int_{S^{n-1}_{r(t)}\times[-1,1]}f(t)=0
\end{equation}
for all $t\in(0,1)$, and that $\abs{f(t)}\leq 1$ for $t\in(0,1/2)$.
Then
\begin{equation}
 \inf_{T\in\mathbb{R}^n} \sup_{C_t}|f(t)-f_{T}|\leq D(1-t)^{\tfrac{1}{2}+\tfrac{1}{n-1}}
\end{equation}
for all $t\in[1/2,1)$, where $D<\infty$ is a constant.
\end{prop}

\begin{proof} Since the family of functions $f$ is invariant under translations, the Laplacian scales like one over distance squared, and $\abs{A_{C_t}}^2=\frac{n-1}{r(t)^2}$,
we can identify $f$ with a family of functions  $\tilde f=\{\tilde f(t)\}_{t\in(0,1)}$ on the unit sphere $S^{n-1}$ satisfying the parabolic equation
\begin{align}
\partial_t\tilde f = \frac{1}{{2(n-1)(1-t)}} \left( \Delta_{\SS^{n-1}} + n-1 \right)\tilde f.
\end{align}
The other assumptions then read $\int_{S^{n-1}}\tilde{f}(t)=0$ for all $t\in(0,1)$, and $|{\tilde{f}(t)}|\leq 1$ for $t\in(0,1/2)$.

The eigenvalues of $-\Delta_{\SS^{n-1}}$ are $\lambda_\ell =\ell(\ell + n-2)$, so $\lambda_{0} =0,\lambda_{1} = n-1, \lambda_{2} = 2n,\dots$.
If we write $\tilde f = \sum \gamma_{j}(t)\varphi_{j}(x)$, then
\begin{equation}
\frac{d}{dt} \gamma_{j} = \frac{1}{{2(n-1)(1-t)}} (-\lambda_{j}+n-1)\gamma_j,
\end{equation}
which has the solution
\begin{equation}
\gamma_{j}(t) = \gamma_{j}(0)(1-t)^{\frac{\lambda_{j}-n+1}{2(n-1)}}.
\end{equation}
In particular,
\begin{equation}
\gamma_{2}(t) = \gamma_{2}(0) (1-t)^{\frac 12 + \frac{1}{n-1}}.
\end{equation}
The assertion follows.
\end{proof}

\section{Blowdown analysis}\label{sec_blowdown}

Let $G(x_1,\ldots,x_{n+1}):=x_1\partial_{x_2}-x_2\partial_{x_1}$, and let $\mathcal{R}$ be the set of all $A_\ast G$, where $A\in SO_n\subset SO_{n+1}$.
For an arbitrary center $\bar{x} \in \RR^{n}\subset\RR^{n+1}$, we define $\mathcal{R}_{\bar{x}}:=\{R(\cdot-\bar{x})|R\in \mathcal{R}\}$.
Consider the function
\begin{equation}
B(h) : = \inf_{x\in\RR^{n}}\sup_{R\in\mathcal{R}_x}\sup_{\{z=h\}}|f_{R}|.
\end{equation}
Note that
\begin{equation}
 G_{\bar{x}}(x):=G(x-\bar{x})=(x_1-\bar{x}_1)\partial_{x_2}-(x_2-\bar{x}_2)\partial_{x_1}=G(x)-(\bar{x_1}\partial_{x_2}-\bar{x_2}\partial_{x_1}).
\end{equation}
More generally, if we let $R_{\bar{x}}(x):=R(x-\bar{x})$ and $A_{\bar{x}}(x):=A(x-\bar{x})+\bar{x}$ for  $R\in\mathcal{R}$ and $A\in SO_n$,
then
\begin{equation}\label{eq_rot_trans}
 A_{\bar{x}}G_{\bar{x}}=AG-T_{A,\bar{x}},
\end{equation}
where $T_{A,\bar{x}}$ is translation along the vector $A\bar{x}+AS\bar{x}-\bar{x}$ with $S\bar{x}=(-\bar{x}_2,\bar{x}_1,0,\ldots,0)^t$. For a dense set of $A\in SO_n$, the matrix
$A(I+S)-I$ is invertible. We conclude that
\begin{align}
B(h) = \inf_{T\in\RR^n}\sup_{R\in\mathcal{R}} \sup_{\{z=h\}}|f_{R}-f_{T}|.
\end{align}

\noindent\underline{Case 1:} Suppose there is a sequence $h_{m}\to \infty$ with $B(h_{m}) = 0$. For each $m$, choose $x_m$ such that
\begin{equation}
 \sup_{\{z=h_m\}}|f_{R_{x_m}}|=0
\end{equation}
for every $R_{x_{m}}\in \mathcal{R}_{x_m}$. Proposition \ref{prop_ell_est} implies that $f_{R_{x_m}}=0$ in the region $\{z\leq h_m\}$.
By equation \eqref{eq_rot_trans}, we can write $f_{R_{x_m}} = f_{R} - f_{T_m}$ for some $R\in \mathcal{R}$ and some translation $T_m$ (depending on $R_{x_m}$ and $x_m$).
Then, $f_{R} - f_{T_m}=0$ in the region $\{z\leq h_m\}$. Thus, $x_m$ is constant and $f_{R_{x_m}}=0$ everywhere for every $R_{x_{m}}\in \mathcal{R}_{x_m}$, and hence we have proven rotational symmetry.\\

\noindent\underline{Case 2:} Suppose now that $B(h) > 0$ for $h$ large. Fix $\tau \in (0,\tfrac{1}{2})$ such that
$\tau^{-\tfrac{1}{n-1}}>8D$, where $D$ is the constant from Proposition \ref{lemma_decay}.
Since $B(h) \leq O(h^{1/2})$,\footnote{Using the maximum principle, we can prove that the tensor $S = (n-1)A- H g +Hdz\otimes dz$ satisfies
 $\abs{S}\leq O(h^{-1+\varepsilon})$ for any $\varepsilon>0$, which implies the better estimate $B(h) \leq O(h^{\varepsilon})$.
However, the rough estimate $B(h) \leq O(h^{1/2})$, which follows from being asymptotically cylindrical, turns out to be good enough for our purpose.} we can then find $h_{m}\to\infty$ such that
\begin{equation}
B(h_{m}) \leq 2 \tau^{-1/2} B(\tau h_{m}).
\end{equation}
Choose $x_{m}$ such that 
\begin{equation}
\sup_{R\in{\mathcal{R}_{{x_m}}}}\sup_{\{z=h_m\}}|f_{R}| = B(h_{m}).
\end{equation}
Choose $R_m\in\mathcal{R}_{x_m}$ such that
\begin{equation}
B(\tau h_m)=\inf_T\sup_{\{z\in \tau h_m\}}\abs{f_{R_m}-f_T}.
\end{equation}
The function $\tilde f_{m} : = \frac{1}{B(h_{m})}f_{R_m}$ satisfies the elliptic equation
\begin{equation}\label{ellip_fm}
 (\Delta_z+\abs{A}^2)\tilde f_m=0.
\end{equation}
Applying Proposition \ref{prop_ell_est} and using the asymptotics for $H$, we obtain
\begin{equation}\label{bounds_fm}
 \sup_{\{z=h\}}\abs{\tilde f_m}\leq 2 \left(\frac{h_m}{h}\right)^{1/2}
\end{equation}
for $h\leq h_m$ and $m$ large.\\

Recall that the family
$\{M_t=M_0+tV\}_{t\in\RR}$ moves by mean curvature flow.
If we view $\tilde{f}_m$ as a one parameter family of functions on $M_t$, then the drift term in $\Delta_z=\Delta+\nabla_{\bar{V}}$ becomes a time derivative, and equation \eqref{ellip_fm} takes the form
\begin{equation}
 \partial_t \tilde{f}_m=(\Delta+\abs{A}^2)\tilde{f}_m.
\end{equation}
Let $p_m=(0,h_m)\in\mathbb{R}^{n+1}$.
Consider the parabolic rescaling $(x,t)\mapsto (\lambda_m(x-p_m),\lambda_m^2t)$, where $\lambda_m=h_m^{-1/2}$. In other words, let
\begin{equation}
\hat{M}_t^m=\lambda_m (M_{\lambda_m^{-2}t}-p_m)
\end{equation}
and
\begin{equation}
\hat{f}_m(x,t)=\tilde{f}_m(\lambda_m^{-1}x+p_m,\lambda_m^{-2}t)
\end{equation}
for $x\in \hat{M}_t^m$. Note that $\hat{M}_t^m$ moves by mean curvature flow and that $\hat{f}_m$ satisfies the parabolic equation
\begin{equation}
\partial_t \hat{f}_m= (\Delta+\abs{A}^2)\hat{f}_m.
\end{equation}
By Proposition \ref{prop_asympt}, for $m\to \infty$ the mean curvature flows $\hat{M}_t^m$ converge to the family of shrinking cylinders
\begin{equation}
C_t=S^{n-1}_{r(t)}\times \mathbb{R},
\end{equation}
where $r(t)=\sqrt{2(n-1)(1-t)}$. Using the estimate \eqref{bounds_fm} we obtain
\begin{equation}
\limsup_{m\to\infty} \sup_{t\in[\delta,1-\delta]}\sup_{\abs{z}\leq \delta^{-1}}\abs{\hat{f}_m}<\infty
\end{equation}
for any given $\delta\in(0,1/2)$.
Hence, the functions $\hat{f}_m$ converge (subsequentially) to a family of functions $f=\{f(t)\}_{t\in(0,1)}$ on $C_t$ that satisfy
\begin{equation}
 \partial_t f=(\Delta_{C_t}+\abs{A_{C_t}}^2) f.
\end{equation}
The limit $f$ is invariant under translations and the estimate \eqref{bounds_fm} implies that $\abs{f(t)}\leq 4$ for $t\in(0,\tfrac12)$.
Moreover, since $\textrm{div}_{\mathbb{R}^{n+1}} R_{m}=0$ and $\bangle{R_{m},V}=0$, the divergence theorem yields that
\begin{equation}
 \int_{\{h_1\leq z\leq h_2\}}\bangle{R_{m},\nu}=0
\end{equation}
on $M$ for any $h_1<h_2$. Thus, $f$ also satisfies assumption \eqref{decay_assump_average}.
Hence, Proposition \ref{lemma_decay} implies that
\begin{equation}\label{eq_contr1}
 \inf_{T} \sup_{C_t}|f(t)-f_{T}|\leq 4D(1-t)^{\tfrac{1}{2}+\tfrac{1}{n-1}}
\end{equation}
for all $t\in [1/2,1)$.
On the other hand, we have
\begin{align}
 \inf_{T} \sup_{\{z=0\}} | \hat{f}_m(1-\tau)-f_T|
& =\inf_{T} \sup_{\{z = h_{m}\}} | \tilde{f}_m(h_m(1-\tau))-f_T| 
 =\inf_{T} \sup_{\{z = \tau h_{m}\}} | \tilde{f}_m(0)-f_T| \\
& = \frac{1}{B(h_{m})} \inf_{T} \sup_{\{z =  \tau h_{m}\}} |f_{R_{m}}-f_T| 
 = \frac{B(\tau h_{m})}{B(h_{m})} 
 \geq \frac 12 \tau^{1/2}.\nonumber
\end{align}
Taking the limit as $m\to\infty$ gives
\begin{equation}\label{eq_contr2}
 \inf_{T} \sup_{C_{1-\tau}}|f(1-\tau)-f_{T}|\geq \frac12 \tau^{1/2}.
\end{equation}
Since $\tau^{-\tfrac{1}{n-1}}>8D$ the inequalities \eqref{eq_contr1} and \eqref{eq_contr2} are in contradiction. This completes the proof.

\bibliography{bowl_uniqueness}

\bibliographystyle{alpha}

\vspace{10mm}

{\sc Courant Institute of Mathematical Sciences, New York University, 251 Mercer Street, New York, NY 10012}\\

\emph{E-mail:} robert.haslhofer@cims.nyu.edu

\end{document}